\documentclass[12pt]{article}
\usepackage{macros}

\date{}
\title{Two Hilbert schemes in computer vision}
\author{Max Lieblich and Lucas Van Meter}

\begin{document}
\maketitle
\begin{abstract}
    We study multiview moduli problems that arise in computer
  vision. We show that these moduli spaces are always smooth and
  irreducible, in both the calibrated and uncalibrated cases, for any
  number of views. We also show that these moduli spaces always admit
  open immersions into Hilbert schemes for more than two views,
  extending and refining work of Aholt--Sturmfels--Thomas. We use
  these moduli spaces to study and extend the classical twisted pair
  covering of the essential variety.
  \end{abstract}

\section{Introduction}
\label{sec:intro}

In this paper, we discuss a functorial approach to multiview geometry, a
subfield of computer vision. The literature on multiview geometry is vast,
although this is the first attempt that we know of to use the techniques of
modern functorial algebraic geometry to approach the subject. As we hope to
demonstrate here and elsewhere, this approach has a great deal of promise. A
beautiful introduction to the subject can be found in
\cite{hartley2003multiple}. Earlier versions of this paper (available on the
arxiv) also contain a condensed introduction to the subject suitable for
algebraic geometers.

\subsection{Our results}
\label{sec:our-results}
The main result of this paper is the following, proven in
\cref{sec:defos-moduli,sec:comparisons}.

\begin{thm}\label{T:main-uncal}
  There are smooth irreducible varieties $\Cam_n$ and $\CalCam_n$ parametrizing
  $n$-view camera configurations and $n$-view calibrated camera
  configurations, respectively.
  \begin{enumerate}[leftmargin=1.5cm]
  \item The variety $\Cam_n$ has dimension $11n-15$. For all $n>1$, sending a
configuration to its joint image defines a locally closed embedding
$$\Cam_n\hookrightarrow\Hilb_{(\bP^2)^n}.$$ If $n>2$ then this morphism is an
open immersion,  
so that $\Cam_n$ is identified with an open subscheme of the smooth
locus of $\Hilb_{(\bP^2)^n}$. 
  \item The variety $\CalCam_n$ has dimension $6n-7$. For all $n>1$, there is a
  natural locally closed embedding
  $$\CalCam_n\hookrightarrow\Hilb_{C_1\times\cdots\times
    C_n\subset(\bP^2)^n}$$ (where the latter is a diagram Hilbert
  scheme; see 
\cref{sec:hilbert}). If $n>2$ then this morphism is an open immersion.

  \item The natural decalibration morphism $\nu_n:\CalCam_n\to\Cam_n$ is finite,
    proper and unramified. The morphism $\nu_2$ is an étale cover of its image
    with general fiber of order $2$. For $n>2$ the morphism $\nu_n$ is
    generically injective but not injective.
  \end{enumerate}
\end{thm}

The statements on Hilbert schemes
generalize and refine the results of \cite{aholt2011hilbert}. In particular, our
methods show that the formation of the multiview variety gives an open immersion
into the Hilbert scheme at all points, identifying the moduli space with an open
subscheme of the Hilbert scheme.

\subsection{Methodological contributions}
\label{sec:meth-contr}

There are a few basic principles that set this work apart from other
work on multiview geometry.

\begin{enumerate}[leftmargin=1.5cm]
\item The \emph{functorial method\/}, common in modern algebraic
  geometry, gives us insight into the intrinsic geometry of 
  natural moduli problems growing out of the classical
  constructions. While \cite{aholt2011hilbert} uses the GIT quotient to 
construct the moduli of uncalibrated camera configurations, this method does not 
obviously generalize to a construction for calibrated cameras. Additionally, by 
developing the functorial theory of cameras we hope to make the field of 
multiview geometry accessible to a wider audience in pure mathematics.
  
\item The \emph{geometric view of calibration\/} via calibration data
  gives us insight into the structure of the space of calibrated
  cameras in a way that seems not to have been considered before. In
  particular, by restricting camera configurations to morphisms
  between calibrating conics, we get a fibration structure on the
  moduli space of calibrated camera configurations that is quite useful for
  studying the moduli space. In \cref{sec:mod-cal}, there's a
  third Hilbert scheme -- the Hilbert scheme of the product of
  calibrating conics -- that is the base of this fibration. This way
  of thinking about calibration can also be used to understand the
  essential variety in new ways. In \cite{essvar}, this is used to
  reproduce results of both \cite{demazure:inria-00075672} and
  \cite{kileel} (which itself used the results of
  \cite{demazure:inria-00075672}) from first principles, among other things.
  \item The use of \emph{diagram Hilbert schemes\/} allows us to treat the case 
of calibrated cameras similarly to how uncalibrated cameras are treated in 
\cite{aholt2011hilbert}. Instead of closed subschemes, as were used for the 
calibrated case, we use a type of flag to keep track of the calibration data. 
This transparently recovers the result that the moduli space is
  open in a Hilbert scheme.
\end{enumerate}

This paper also opens up many new lines of inquiry and leaves many questions 
unanswered. We discuss a few of these questions in \cref{sec:questions}.



\subsection*{Acknowledgments}
\label{sec:acknowledgments}

We had interesting and helpful conversations with many people during the course of this work: Sameer Agarwal, Roya Beheshti, Dustin Cartwright, Charles Godfrey, Richard
Hartley, Jonathan Hauenstein, Fredrik Kahl, Joe Kileel,
Irina Kogan, Luke Oeding, Peter Olver,
Brian Osserman, Tomas Pajdla, Jean Ponce, Jessica Sidman, Bernd
Sturmfels, Rekha Thomas, Matthew Trager, and Bianca Viray. Rekha Thomas
gave especially valuable remarks that helped us significantly improve
our exposition. We were
partially supported by NSF grants CAREER DMS-1056129 and DMS-1600813
during the preparation of this paper. We benefitted greatly from
the Berlin Algebraic Vision meeting in October of 2015, hosted at TU Berlin with
support from the Einstein Center for Mathematics, DFG Priority
Project SPP 1489, and the NSF, and the AIM
meeting on Algebraic Vision in May 2016, with support from AIM and the
NSF.

We thank the referees and editors for patiently giving us numerous helpful
suggestions and comments.

\section{The algebraic geometry of pinhole cameras}
\label{sec:cameras}

In this section we review the basic theory of pinhole cameras, with a geometric
emphasis. We include a canonical treatment of calibrated cameras with
a greater focus on the geometry of the calibrating conics. 
For the sake of clarity, we focus in 
\cref{sec:definitions} and \cref{sec:multiview} on the geometry over an
algebraically closed field. In \cref{sec:rel} we study what
happens over a general base scheme, as a preparation for the study of moduli
and deformation theory in \cref{sec:defos-moduli}.

\subsection{Basic definitions}
\label{sec:definitions}

\begin{defn}\label{def:camera}
  A \emph{pinhole camera\/} is a surjective rational map
$\varphi:\bP^3\dashrightarrow\bP^2$ given by three linearly
independent sections of $\ms O_{\bP^3}(1)$. The \emph{center\/} of the
camera is the unique point $p\in\bP^3$ at which $\varphi$ is undefined.
\end{defn}


\begin{defn}\label{def:calib-plane}
  A \emph{calibrated plane\/} is a pair $(\bP^2,D)$ with $D$ a smooth conic.
\end{defn}

\begin{defn}\label{defn:calibration-datum}
  A \emph{calibration datum\/} for a pinhole camera $\phi$ is a pair
  of planar degree $2$ curves $C\subset\bP^3$ and $D\subset\bP^2$ such that
  $D$ is a smooth conic and the restriction $\phi_C:C\dashrightarrow\bP^2$ 
factors through the inclusion $D\subset\bP^2$.

  If $C$ is smooth, the calibration datum will be called \emph{smooth\/} or
  \emph{non-degerate\/}; otherwise it will be called \emph{degenerate\/}. If a
  calibrated plane $(\bP^2,D)$ is fixed, a \emph{relative calibration datum\/}
  for a pinhole camera $\bPhi$ is a curve $C\subset\bP^3$ such that $(C,D)$ is a
  calibration datum for $\bPhi$.
\end{defn}

\begin{remark}
  If $C$ is smooth then it follows from the linearity of the camera
  projection that $\bPhi$ must map $C$ isomorphically to $D$, and that 
the center of $\bPhi$ is not contained in the plane spanned by $C$. If
$C$ is degenerate, it must be a divisor-theoretic sum of two lines on
the quadric cone in  $\bP^3$ generated by $D$ under the projection $\bPhi$ 
(i.e., a union of two distinct rulings or a double ruling). When the two cone points are distinct (i.e., the configuration is general), a union of two distinct rulings cannot occur as a limit of calibration data.
\end{remark}

\begin{remark}
A given camera with calibrated image plane $(\bP^2,D)$ has infinitely many
relative calibration data: one can take any plane section of the
quadric cone in $\bP^3$ lying over $D$. Once we look at configurations
of two or more cameras, there will be at most two calibration data
(smooth or degenerate). This is described at length in 
\cref{sec:geom-map}.

Degenerate calibrations give us closures of natural moduli
spaces, including the closure of the classical twisted pair moduli
space $\SO(3)\times\bP^2$ to a finite \'etale cover of the essential
variety described in \cref{sec:twisted-pairs-intro}. Imagining
the system of plane sections of the cone over $D$, one readily sees
that degenerate calibration data arise as limits of smooth calibration data.
\end{remark}

\begin{defn}\label{defn:calibrated-camera} A \emph{calibrated
camera\/} is a pair $(\varphi, (C, D))$ where $\varphi$ is a pinhole
camera and $(C,D)$ is a calibration datum for $\varphi$.
\end{defn}


\begin{remark} In the classical literature, a camera is called
\emph{calibrated\/} (or sometimes \emph{normalized\/}) when it takes the 
absolute conic to the Euclidean conic: more precisely, we can endow $\bP^3$ with 
coordinates $x,y,z,w$ and $\bP^2$ with coordinates $X,Y,Z$, and then we take the 
curves $C$ and $D$ to be given by the equations $\{w=0, x^2+y^2+z^2=0\}$ and
$\{X^2+Y^2+Z^2=0\}$, respectively. Note that any camera as
described here with a smooth calibration datum can be transformed to a
classically calibrated camera 
by applying suitable automorphisms to $\bP^3$ and $\bP^2$. (This is
not unique.) The degenerate calibrations cannot.

There are two reasons to use this more flexible approach: 
\begin{enumerate}[leftmargin=1.5cm]
\item[(1)] It leads to the ``right definition'' of the moduli space of 
calibrated camera configurations (\cref{sec:mod-cal}).
\item[(2)] By always
forcing the absolute conic to map to the Euclidean conic, one makes it
impossible to study modular boundary points where the absolute conic is
flattened until it collapses (yielding degenerate calibrations). As we
will describe below, these degenerate calibrations give geometrically
meaningful compactifications of the space of calibrated camera
configurations. 
\end{enumerate}
\end{remark}

\subsection{Multiview configurations}
\label{sec:multiview}

In this section, we describe some of the geometry attached to a
collection of cameras with distinct centers.

\subsubsection{Uncalibrated cameras}
\label{sec:uncal-config}

\begin{defn} A \emph{multiview configuration\/} is a collection of
cameras $$\phi_1,\ldots,\phi_n:\bP^3\dashrightarrow\bP^2.$$
\end{defn}

\begin{notation} We will generally use
$\bPhi:\bP^3\dashrightarrow(\bP^2)^n$ to denote a multiview
configuration, writing $\bPhi_i=\pr_i\circ\bPhi$ for its components
when necessary.  The \emph{length\/} of $\bPhi$ is the number of
cameras; we will denote it $\length(\bPhi)$.  Write
$\Center(\bPhi)\subset\bP^3$ for the tuple of camera centers. Write
$\pi:\Res(\bPhi)\to\bP^3$ for the blowup of $\bP^3$ at the reduced
closed subscheme supported at the camera centers; if two cameras have
the same center we only count it once. Given an index $i$, let $E_i$
denote the exceptional divisor over the $i$th camera center, with
canonical inclusion $\iota_i:E_i\hookrightarrow\Res(\bPhi)$. By the
previous convention, this means that there can be $i\neq j$ for which
$E_i=E_j$.
\end{notation}

\begin{defn}\label{defn:general} A multiview configuration $\bPhi$ is
\emph{general\/} if the camera centers are all distinct. It is
\emph{non-collinear\/} if the camera centers do not all lie on a
single line, and \emph{collinear\/} otherwise.
\end{defn}

\begin{defn}\label{defn:config-isom} An \emph{isomorphism\/} between
multiview configurations $\bPhi^1$ and $\bPhi^2$ of common length $n$
is an automorphism $\varepsilon:\bP^3\to\bP^3$ fitting into a
commutative diagram
        $$
        \begin{tikzcd} \bP^3\ar[dr, dashed, "\bPhi^1"]\ar[dd,
"\varepsilon"'] & \\ & (\bP^2)^n \\ \bP^3\ar[ur, dashed, "\bPhi^2"']
        \end{tikzcd}
        $$
\end{defn}


\begin{lem}\label{L:extendy} Let $Y$ be a scheme, and let $(\sL,s_0,\dots,s_n)$
be an invertible sheaf with $n$ sections. If $Z$ is the zero scheme of
$s_0,\ldots,s_n$ 
then the rational map induced by this linear series extends uniquely
to a morphism $\Bl_Z Y \to \bP^n$.
\end{lem}
        
\begin{proof} By definition the sections $s_0,\ldots,s_n$ define a surjection
          $$\ms O_Y^{n+1}\twoheadrightarrow\ms L\tensor\ms I_Z,$$
which extends to a surjective map of $\ms O_Y$-algebras
$$\Sym^{\ast}(\ms L^\vee)^{\oplus n+1}\twoheadrightarrow\bigoplus\ms I^n.$$
The induced map on relative $\bProj$ constructions gives the desired morphism.
\end{proof}

\begin{prop}\label{prop:resolution-of-config} Given a multiview
configuration $\bPhi$, there is a unique commutative diagram
        $$
        \begin{tikzcd} \Res(\bPhi)\ar[dr, "\rho"] & \\ &
(\bP^2)^{\length(\bPhi)}\\ \bP^3.\ar[uu, dashed, "\pi^{-1}"]\ar[ur,
dashed, "\bPhi"']
        \end{tikzcd}
        $$
        The diagram has the property that for each $i$, the
composition
        $$
        \begin{tikzcd} E_i\ar[r, "\iota_i"] & \Res(\bPhi)\ar[r,
"\rho"] & (\bP^2)^{\length(\Phi)}\ar[r,"\pr_i"] & \bP^2
        \end{tikzcd}
        $$
        is an isomorphism.
\end{prop}
\begin{proof} \cref{L:extendy} shows the existence and
uniqueness of the desired diagram. To check that the composition is an
isomorphism on exceptional divisors one can see that each map is
locally isomorphic to the morphism $\Bl_{0}\bA^3\to\bP^2$ that
resolves the canonical presentation $\bA^3\setminus\{0\}\to\bP^2$, and
here one can simply check that the induced map from the exceptional
divisor to the plane is an isomorphism. We omit the details.
\end{proof}

\subsubsection{Calibrated cameras}
\label{sec:cal-config}

When the cameras are adorned with calibration data, we track these
data through the diagrams.

\begin{defn} Given a multiview configuration
$\bPhi:\bP^3\dashrightarrow(\bP^2)^n$, a \emph{multiview
calibration datum\/} is a pair $(C, (C_1,\ldots,C_n))$ such that for
each $i=1,\ldots,n$ the pair $(C,C_i)$ is a calibration datum for
$\bPhi_i$. Given a tuple of calibrated planes $(\bP^2,C_i)$ for
$i=1,\ldots,n$, a \emph{relative calibration datum\/} for
$\bPhi$ is a curve $C\subset\bP^3$ such that
$(C,(C_1,\ldots,C_n))$ is a calibration datum for $\bPhi$.
\end{defn}

\begin{notation} We will write $\bC$ for a calibration datum $(C,
(C_i))$, and then $\bC_0=C$ and $\bC_i=C_i$ for $i=1,\ldots,n$.
\end{notation}

\begin{notation}
  A calibrated multiview configuration $(\bPhi,\bC)$ will be called
  \emph{non-degenerate\/} if the calibration datum is non-degenerate.
\end{notation}

\begin{defn} An \emph{isomorphism\/} between multiview configurations
with calibration data $(\bPhi^1, \bC^1)$ and $(\bPhi^2,\bC^2)$ of
common length $n$ is an isomorphism $\eps:\bPhi^1\to\bPhi^2$ of
multiview configurations as in \cref{defn:config-isom} such
that $\eps(\bC^1_0)=\bC^2_0$ and such that for $i=1,\ldots,n$ we have
$\bC^1_i=\bC^2_i$.
\end{defn}


\subsubsection{A characterization of isomorphic general
configurations}
\label{sec:characterization-of-isom-configs}

In this section we briefly consider when two multiview configurations
$\bPhi^1$ and $\bPhi^2$ are isomorphic (and similarly when they are
endowed with calibration data). This will play a role in studying a
particular map from the moduli space to Hilbert schemes in later
sections of this paper.

\begin{defn} Given a multiview configuration $\bPhi$, the
\emph{associated multiview scheme\/}, also known as the \emph{joint
image\/} \cite{aholt2011hilbert,trager2015joint}, is the
scheme-theoretic image of the resolution 
$\Res(\bPhi)$ under the canonical extension $\rho$ of 
\cref{prop:resolution-of-config}. It is denoted
$\Scheme(\bPhi)$. Working over a field (as we temporarily are here), the multiview
scheme is a variety, and is called the ``multiview variety'' in
\cite{aholt2011hilbert}. 
\end{defn}

In the following, an \emph{$n$-term flag of schemes\/} will be a sequence of 
closed immersions 
$$X_0\hookrightarrow X_1\hookrightarrow X_2\hookrightarrow\cdots\hookrightarrow 
X_{n-1}.$$

\begin{defn} Given a calibrated multiview configuration $(\bPhi,C)$
  with calibrated image planes
  $(\bP^2,C_i)$, $i=1,\ldots,n$,
 the \emph{associated multiview flag\/}, denoted $\Flag(\bPhi,C)$, is the 
$2$-term flag of schemes
 $C\subset\Scheme(\bPhi)$ contained in $C_1\times\cdots\times
 C_n\subset (\bP^2)^n$.
\end{defn}
       
        As we will gradually see, the following lemma is the key
result connecting the abstract moduli problems we study here to
Hilbert schemes.
        
\begin{lem}\label{lem:surj-structure-morph} 
  The canonical map $\sO_{\Scheme(\bPhi)} \to \R\rho_\ast
\sO_{\Res(\bPhi)}$ is a quasi-isomorphism. Equivalently, the canonical map $\rho^\sharp:\ms
O_{(\bP^2)^n}\to\rho_\ast\ms O_{\Res(\bPhi)}$ is an isomorphism and 
all higher direct images $\R^i\rho_\ast\ms O_{\Res(\bPhi)}$ (with
$i>0$) vanish.
\end{lem}
\begin{proof}
For the first statement, note that $\rho_\ast\ms
O_{\Res(\bPhi)}$ is a finite $\ms O_{(\bP^2)^n}$-algebra by
properness. Moreover, since every non-empty fiber of $\rho$ is
geometrically integral (it being an intersection of lines, hence
either a point or a line), we see that $\rho^{\sharp}$ is surjective
after base change to any point of $(\bP^2)^n$. By Nakayama's lemma,
$\rho^\sharp$ is surjective.

Now we show that the higher direct images vanish.  By the Theorem on
Formal Functions \cite[Théorème 4.1.5]{MR0217085}, the completion of
$\R^i\rho_\ast\ms O$ at a point $p$ is isomorphic to $\lim\H^i(X_m,\ms
O_{X_m})$, where $X_m$ is the $m$th infinitesimal neighborhood of the
fiber of $\rho$ over $p$. When the fiber is empty or a point, this
vanishes. The only interesting case is the unique singular point that
is the image of the strict transform of the line through all camera
centers, in the collinear case. Note that $\sO_{X_m}$ is filtered by
subquotients that are symmetric powers of the ideal sheaf $\ms
I_{X_0}$ restricted to $X_0$. Given a line $L$ in $\bP^3$, we have
that $\ms I_L|_L\cong\ms O_L(-1)^{\oplus 2}$. For each point on $L$
that we blow up, the ideal sheaf gets twisted by $1$ (functions from
$\bP^3$ vanish to extra order on the strict transform along the
intersection with the exceptional divisor). In fact, if we are blowing
up $n$ points, we have that $\ms I_{X_0}|_{X_0}\cong\ms
O_{X_0}(n-1)^{\oplus 2}$. The $\ell$th symmetric power will be a sum
of copies of $\ms O_{X_0}(\ell(n-1))$. All such sheaves
have vanishing $\H^i$ for all $i>0$. 

 Write $\ms I_m$ for the ideal sheaf of $X_m$ in $\Res(\bPhi)$.
Consider the standard exact sequences
$$0\to \ms I_{m-1}/
\ms I_m\to\ms O_{X_m}\to\ms O_{X_{m-1}}\to 0.$$
The above calculations show inductively that $\H^i(X_n,\ms O_{X_n})=0$
for all $n\geq 0$ and all $i>0$. This concludes the proof.
\end{proof}

\begin{cor}\label{prop:res-closed} If $\Phi$ is a
non-collinear multiview configuration then the map $\rho: \Res(\Phi)
\to (\bP^2)^n$ is a closed immersion.
\end{cor}
        
\begin{proof} By the non-collinearity assumption, the
geometric fibers of $\rho$ all have length at most 1. Thus, $\rho$ is
proper and quasi-finite, hence finite. Applying 
\cref{lem:surj-structure-morph} then shows that $\rho$ is a closed
immersion.
\end{proof}

\begin{lem}\label{lem:stupid-extend} Suppose
$\phi_1,\phi_2:\bP^3\dashrightarrow\bP^2$ are cameras and
$\alpha:\bP^3\dashrightarrow\bP^3$ is a birational automorphism such
that $\phi_2=\phi_1\circ\alpha$. If $\alpha$ and $\phi_1\circ\alpha$
are both regular on an open subset $U\subset\bP^3$ whose complement
has codimension at least $2$ then $\alpha$ extends to a unique
regular automorphism $\bP^3\to\bP^3$.
\end{lem}
\begin{proof} Removing the center of $\phi_1$ if necessary, we
may assume that there is an open subscheme $U\subset\bP^3$ on which $\phi_1$, 
$\phi_2$,
and $\alpha$ are all regular and $\codim(\bP^3, \bP^3\setminus U)\geq
2$. By assumption, $\phi_i^\ast\ms O(1)=\ms O_U(1)$. Thus,
$\alpha^\ast\ms O(1)=\ms O(1)$. Since $\Gamma(U,\ms
O(1))=\Gamma(\bP^3,\ms O(1))$, we conclude from the universal property
of projective space that the morphism $\alpha:U\to\bP^3$ extends to a
unique endomorphism $\widetilde{\alpha}$ of $\bP^3$. Since $\alpha$ is
birational, $\widetilde{\alpha}$ is an isomorphism, as desired.
\end{proof}

\begin{prop}\label{prop:config-isom-image} Two multiview
configurations $\bPhi^1$ and $\bPhi^2$ of length $n$ are isomorphic if
and only if their associated multiview schemes in $(\bP^2)^n$ are
equal. Two calibrated multiview configurations $(\bPhi^1, C_1)$ and 
$(\bPhi^2,C_2)$ are isomorphic if and only if their associated multiview flags 
$\Flag(\bPhi^1,C_1)$ and $\Flag(\bPhi^2,C_2)$ are equal.
\end{prop}
\begin{proof} Since $\bPhi^i$ is birational onto its image for $i=1,2$, we see
that if $\Scheme(\bPhi^1)=\Scheme(\bPhi^2)$ then there is a birational
automorphism $\alpha:\bP^3\dashrightarrow\bP^3$ such that
$\bPhi^2=\bPhi^1\circ\alpha$. Moreover, $\pr_1\circ\bPhi^1$, $\alpha$, and
$\pr_1\circ\bPhi^2\circ\alpha$ are all regular on the open subscheme of $\bP^3$
that is the complement of the line joining the centers of the two cameras 
$\pr_1\circ\bPhi^1$
and $\pr_1\circ\bPhi^2$ (as this maps isomorphically into the smooth locus of
$\Scheme(\bPhi^1)$). Applying \cref{lem:stupid-extend}, we see thats
$\alpha$ is regular, as desired. The calibrated case follows, once we note that 
the calibrating curves lie in the regular locus of all cameras.
\end{proof}
 
\subsection{Relativization}
\label{sec:rel}

In this section we describe how to generalize the results of
\cref{sec:definitions,sec:multiview} to families of cameras over an arbitrary
base space. This is a necessary step towards defining the moduli of camera
configurations. 

\begin{defn}\label{defn:relative-pinhole-cam} Given a scheme $S$, a
\emph{relative pinhole camera\/} over $S$ is a rational map 
$p:\bP\dashrightarrow\bP^2_S$ over $S$ uniquely determined by the
following information:
        \begin{enumerate}[leftmargin=1.5cm]
                \item the scheme $\bP$ is a Zariski $\bP^3_S$-bundle (i.e., has the form $\bP(V)$ for a locally free $\ms O_S$-module of rank $4$);
                \item there is a map $\sigma:\ms O_{\bP}^{\oplus
3}\to\ms O_{\bP}(1)$ whose cokernel is an invertible sheaf supported
exactly over a section $Z$ of $\bP \to S$, called the \emph{camera
center\/};
                \item a representative of $p$ is given by the morphism
$\bP \setminus Z \to \bP^2_S$ determined by the quotient $\sigma_{\bP
\setminus Z}$ and the universal property of projective space.
        \end{enumerate}
\end{defn}

Throughout this section, when the base scheme $S$ is clear, we will often simply
write $\bP^2$ for $\bP^2_S$, etc.

\begin{defn} Given a scheme $S$, a \emph{relative multiview configuration of
length $n$\/} over $S$ is given by a proper $S$-scheme $\bP\to S$ of finite
presentation and a rational map $\bPhi:\bP \dashrightarrow(\bP^2_S)^n$ over $S$
such that for each $i$ the composition $\pr_i\circ\bPhi$ is a relative pinhole
camera as in \cref{defn:relative-pinhole-cam}.
        
Two relative multiview configurations
\[ \bPhi^i:\bP_i \dashrightarrow \bP^2_S,\quad i=1,2\] 
are \emph{isomorphic\/} if there is an $S$-isomorphism $\eps:
\bP_1 \xrightarrow{\sim} \bP_2$ such that $\bPhi^2=\bPhi^1 \circ
\eps$.
\end{defn}

In what follows, we will write $\bP^2$ for $\bP^2_S$, etc., when the base scheme
$S$ is understood.

\begin{notn} Given a multiview configuration $\bPhi:\bP\to(\bP^2)^n$
of length $n$, we will write
        \begin{enumerate}[leftmargin=1.5cm]
                \item $S(\Phi)$ for the domain $\bP$ of $\bPhi$;
                \item $Z_1(\bPhi),\dotsc,Z_n(\bPhi)\subset\bPhi$ for
the camera centers;
                \item $Z(\bPhi)$ for the scheme-theoretic union
$Z_1(\bPhi)\cup\cdots\cup Z_n(\bPhi)$;
                \item $\Res(\bPhi)$ for the blowup of $S(\bPhi)$ in
$Z$.
        \end{enumerate}
\end{notn}

\begin{defn}\label{defn:general-rel} A relative multiview
configuration $\bPhi$ over $S$ is \emph{general\/} if the camera
centers $Z_1,\ldots,Z_{\length(\bPhi)}$ are pairwise disjoint closed
subschemes of $\bP$.
\end{defn}

\begin{defn} A relative multiview configuration
$\bPhi:\bP\dashrightarrow(\bP^2_S)^n$ over $S$ is \emph{collinear} if
there is a closed subscheme $\bL\subset S(\bPhi)$ that is a relative
line over $S$ and that contains $Z(\bPhi)$. It is
\emph{nowhere-collinear\/} if it is not collinear upon any basechange
$S'\to S$.
\end{defn}


\begin{defn} Given a relative multiview configuration $\bPhi$ of
length $n$ over $S$, a \emph{calibration datum\/} for $\bPhi$
is a pair $(C, (C_1,\ldots,C_n))$ where 
\begin{enumerate}[leftmargin=1.5cm]
  \item $C \subset \bP$ is a relative degree two curve over $S$;
  \item $C_i \subset \bP^2_S$ is a relative
smooth conic over $S$ for $i=1, \ldots, n$;
  \item for $i=1,\ldots, n$, the induced morphism $(\pr_i \circ \bPhi)_C$ factors through $C_i$.
\end{enumerate}
If $C$ is smooth, the calibration datum will be called \emph{smooth\/} or
\emph{non-degenerate\/}; otherwise it will be called \emph{degenerate\/}.
\end{defn}

\begin{prop}\label{prop:resolution-of-config-rel} Given a general relative 
multiview
configuration $\bPhi$ over $S$, there is a unique commutative diagram
        $$
        \begin{tikzcd} \Res(\bPhi)\ar[dr, "\rho"] & \\ &
(\bP^2)^{\length(\bPhi)}\\ \bP.\ar[uu, dashed, "\pi^{-1}"]\ar[ur,
dashed, "\bPhi"']
        \end{tikzcd}
        $$
        The diagram has the property that for each $i$, the
composition
        $$
        \begin{tikzcd} E_i\ar[r, "\iota_i"] & \Res(\bPhi)\ar[r,
"\rho"] & (\bP^2)^{\length(\Phi)}\ar[r,"\pr_i"] & \bP^2
        \end{tikzcd}
        $$
        is an isomorphism. Moreover, this diagram is compatible with
        arbitrary base change on $S$.
\end{prop}

\begin{proof}
  The arrow $\rho$ exists again by \cref{L:extendy}, and the
  functoriality follows from the functoriality of 
 \cref{L:extendy} and the flatness of everything over $S$. Finally,
  the isomorphism condition can be checked on geometric fibers, which
  reduces it to \cref{prop:resolution-of-config}.
\end{proof}

\begin{defn}\label{defn:multiview-scheme} Given a general multiview
configuration $\bPhi$ of length $n$, the scheme-theoretic image of the morphism $\rho$
described in \cref{prop:resolution-of-config-rel} is the \emph{multiview 
scheme\/}
of $\bPhi$. Similarly, given a calibrated multiview configuration 
$(\bPhi,C,(C_1,\ldots,C_n))$, there is an associated flag $\Flag(\bPhi,C)$ 
sitting inside the flag scheme $C_1\times\cdots\times C_n\subset(\bP^2)^n$.
\end{defn}

\begin{notation} The multiview scheme of $\bPhi$ will be denoted
$\Scheme(\bPhi)$. It is a closed subscheme of $(\bP^2)^{\length(\bPhi)}$.
\end{notation}

In the following, we fix conics in $\bP^2$ and only record the curve 
$C\subset\bP^3$ when considering calibrations.
\begin{prop}\label{prop:iso-iso} Two general multiview configurations
$\bPhi^1,\bPhi^2$ of length $n$ over $S$ are isomorphic if and only if
$\Scheme(\bPhi^1)=\Scheme(\bPhi^2)$ as closed subschemes of
$(\bP^2_S)^n$. Similarly, two general calibrated multiview configurations 
$(\bPhi_1,C_1)$ and $(\bPhi_2,C_2)$ are isomorphic if and only if their flags 
$\Flag(\bPhi_1,C_1)$ and $\Flag(\bPhi_2,C_2)$ are equal.
\end{prop} 

The proof of \cref{prop:iso-iso} is a
modification of that of \cref{prop:config-isom-image}. We
require a modification of \cref{lem:stupid-extend}.
\begin{lem} Suppose $A$ is a ring and $U\subset\bP^3_A$ is an open
subset such that for every geometric point $A\to\kappa$ the fiber
$U_{\kappa}\subset\bP^3_\kappa$ has complement of codimension at least
$2$. Suppose $\alpha:U\to\bP^3_A$ is a morphism such that
$\alpha^\ast\ms O(1)=\ms O_U(1)$. Then $\alpha$ extends to a unique
automorphism of $\bP^3_A$.
\end{lem}
\begin{proof} By the universal property of projective space, it
suffices to show that restriction defines an
isomorphism $$\Gamma(\bP^3_A,\ms O(1))\simto\Gamma(U,\ms O(1)).$$ To
show this, it suffices to show that the adjunction map $\nu(1):\ms
O_{\bP^3}(1)\to\iota_\ast\ms O_U(1)$ is an isomorphism of sheaves. By
the projection formula, it suffices to show that the adjunction map
for the structure sheaf
        $$\nu:\ms O_{\bP^3_A}\to\iota_\ast\ms O_U$$
        is an isomorphism. But this is precisely Proposition 3.5 of
\cite{MR2047697}.
\end{proof}

\begin{prop}\label{base-change-image} If $\bPhi$ is a general
multiview configuration over $S$ then for all base changes $T \to S$
we have that the natural morphism
        \[ \Scheme(\bPhi) \times_S T \to \Scheme(\bPhi \times_S T)
        \] is an isomorphism. That is, formation of the associated
multiview scheme is compatible with base change. Furthermore,
$\Scheme(\bPhi)$ is flat over the base.
\end{prop}

\begin{proof} By \cref{lem:surj-structure-morph} the structure
morphism $\ms O_{(\bP^2)^n} \to \rho_* \ms O_{\Res(\bPhi)}$ is
surjective.  Consider the triangle in the derived category
        \[ I \to \ms O_{(\bP^2)^n} \to \RR \rho_* \ms O_{\Res(\bPhi)}
\to I[1].
        \]
        
        Let $i: (\bP^2)^n_q \to (\bP^2)^n$ be an embedding of a
fiber. Pulling back to the fiber and using cohomology and base change
we have
        \begin{align*} \LL i^* \RR \rho_* \ms O_{\Res(\bPhi)} &\simeq
\RR \rho_*\LL i^*_{\Res(\bPhi)} \ms O_{\Res(\bPhi)} \\ &\simeq \RR
\rho_* (\ms O_{\Res(\bPhi)})_q \\ &\simeq (\ms O_{\Res(\bPhi)})_q.
        \end{align*} Applying \cite[ 3.31]{huybrechts2006fourier} to $\RR
\rho_* \ms O_{\Res(\bPhi)}$, we see that it is quasi-isomorphic to a
sheaf flat over the base. But $\mathscr{H}^0(\RR \rho_\ast\ms
O_{\Res(\bPhi)})$ is $\rho_\ast\ms O_{\Res(\bPhi)}$. Thus, we conclude
that the short exact sequence
        \[ 0 \to \ms I \to \ms O_{(\bP^2)^n} \to \rho_* \ms
O_{\Res(\bPhi)} \to 0
        \] consists of $S$-flat sheaves and is compatible with
arbitrary base change.  This establishes the result.
\end{proof}

\section{Moduli and deformation theory}
\label{sec:defos-moduli}

\subsection{Moduli of uncalibrated camera configurations} 
\label{sec:uncal-moduli} In this section we describe the basic moduli
problem attached to uncalibrated camera configurations. In 
\cref{sec:defos-config} we will study the deformation theory of a
configuration $\bPhi$, especially as it relates to the deformation
theory of the associated scheme $\Scheme(\bPhi)$. Ultimately this will allow us 
to embed the moduli space into the Hilbert scheme.

\begin{defn}\label{defn:cam_n} Given a positive integer $n$, the \emph{functor
of camera configurations of length $n$}, denoted $\Cam_n$, has as value over a
scheme $S$ the set of isomorphism classes of general relative multiview
configurations of length $n$.
\end{defn}
Since a camera configuration of length at least $2$ has trivial automorphism
group, it follows from standard descent theory that $\Cam_n$ is a sheaf in the
fppf topology. In this section we will show that it is a quasi-projective
variety.

\begin{notation}\label{notn:quotient}
  Let $\M^n \subset \M_{3 \times 4}^n$ be the locus of $n$-tuples of full rank
  $3 \times 4$ matrices whose kernels are pairwise distinct. Let $T$ be the torus given by the kernel of the multiplication map $\operatorname{\bf G}_m^n\to\operatorname{\bf G}_m$. There is a natural
  free action of $T\times\GL_4$ on $\M^n$ (where the torus $T$ acts diagonally by scaling). Moreover, since $T\times\GL_4$ is 
  reductive over $\Z$ and $M^n_{3\times 4}$ is affine, we can realize the
  quotient sheaf $\M^n/T\times\GL_4$ as an open subvariety of the GIT quotient
  $M^n_{3\times 4}/\!\!/T\times\GL_4$. In particular, the quotient $\M^n/T\times\GL_4$ is a
  smooth quasi-projective variety. Because the action is free, we also know the
  functor of points of $\M^n/T\times\GL_4$: the $S$-valued points are given by pairs
  $(L\to S,\phi:S\to\M^n)$, where $L\to S$ is a $T\times\GL_4$-torsor and $\phi$ is a
  $T\times\GL_4$-equivariant map. In particular, a morphism $\M^n/T\times\GL_4\to Y$ to a
  scheme $Y$ is the same thing as a $T\times\GL_4$-invariant morphism $\M^n\to Y$.
\end{notation}

\begin{prop}\label{prop:quotient description}
   There is a natural isomorphism of functors $c:\M^n/T\times\GL_4\to\Cam_n$.
\end{prop}
\begin{proof}
  Sending a $3\times 4$-matrix to its associated camera defines a morphism
  $\M^n\to\Cam_n$. This is $T\times\GL_4$-equivariant, since, by definition, projective
  automorphisms of $\bP^3$ do not affect the isomorphism class of a camera
  configuration. To see that $c$ is an isomorphism, it suffices to show that
  $c(R)$ is a bijection for any strictly Henselian local ring $R$. In this case,
  every form of $\bP^3$ is trivial, so we see that any camera configuration is
  given by a tuple of matrices, showing that $c$ is surjective. On the other
  hand, by definition, two such configurations are isomorphic if and only if
  they differ by an automorphism of $\bP^3$ and individual scalings of the factors, which says precisely that they lie
  in the same $T\times\GL_4(R)$-orbit in $\M^n(R)$. The result follows.
\end{proof}

\begin{cor}
   If $n>1$ then the space $\Cam_n$ is a smooth quasi-projective scheme over
   $\Spec\bZ$.
\end{cor}
\begin{proof}
This follows immediately from \cref{prop:quotient description} and the remarks in \cref{notn:quotient}.
\end{proof}


\subsection{Deformations of multiview configurations}
\label{sec:defos-config}

In this section, we study the relationship between the infinitesimal deformation
theory of a camera configuration and the deformation theory of its associated
multiview scheme. As we will see in \cref{sec:morph-hilb}, the
deformation-theoretic approach gives strong results on the relationship between
$\Cam_n$ and $\Hilb_{(\bP^2)^n}$, clarifying and improving the groundbreaking
results of \cite{aholt2011hilbert}.  In particular, our infinitesimal analysis
will apply at all points. These methods are very different from the
ideal-theoretic methods of \cite{aholt2011hilbert}. It would be especially
interesting to understand how the cotangent complex argument of
\cref{sec:collinear-def-lift} relates to the Gr\"obner basis calculations in
\cite{aholt2011hilbert}.

\begin{defn}\label{defn:defo} Fix a ring $A$ containing an ideal $I$
such that $I^2=0$ and let $A_0=A/I$.  Suppose $\bPhi^0$ is a relative
multiview configuration of length $n$ over $A_0$. An
\emph{infinitesimal deformation of $\bPhi^0$ to $A$\/} is a pair
$(\bPhi, \eps)$, where $\bPhi$ is a multiview configuration of length
$n$ over $A$ and $\eps:\bPhi\tensor_A A_0\xrightarrow{\sim}\bPhi^0$ is
an isomorphism of relative multiview configurations.
        
An \emph{isomorphism between infinitesimal deformations\/}
$(\bPhi,\eps)$ and $(\bPhi',\eps')$ of $\bPhi^0$ is an isomorphism
$\alpha:\bPhi\xrightarrow{\sim}\bPhi'$ of relative multiview
configurations such that $\eps'\circ\alpha\tensor_A A_0=\eps.$
\end{defn}

\begin{notation}\label{notn:defo funcs}
  We will write $\Def_{\bPhi^0}$ for the functor of isomorphism classes of
  infinitesimal deformations of $\bPhi^0$, and
  $\Def_{\Scheme(\bPhi^0)\subset(\bP^2)^{\length(\bPhi^0)}}$ for the usual
  functor of infinitesimal deformations of the point $[\Scheme(\bPhi^0)]$ of the
  Hilbert scheme $\Hilb_{(\bP^2)^{\length(\bPhi^0)}}$.
\end{notation}

Our goal in this section is to prove the following, which is the key step in our
generalization of the results of \cite{aholt2011hilbert}.

\begin{prop}\label{prop:defo-proj} 
  If $\bPhi$ is a general multiview configuration of length $n>2$
  then the morphism
  $$\Scheme:\Def_{\bPhi^0}\to\Def_{\Scheme(\bPhi^0)\subset(\bP^2)^n}$$
  is an isomorphism of deformation functors.
\end{prop}
\begin{proof}
  The proof will be developed through this section. In particular, the
  injectivity of $\Scheme$ follows from \cref{prop:iso-iso}, and surjectivity
  follows from \cref{prop:def-coll}.
\end{proof}

That is, if $\bPhi$ is a general multiview configuration of length $n>2$ with
associated multiview variety $V\subset(\bP^2)^n$ then, we have that the
infinitesimal deformations of $\bPhi$ are in bijection with the infinitesimal
deformations of $V$ as a closed subscheme of $(\bP^2)^n$. The proof will work
roughly as follows.

\begin{enumerate}[leftmargin=1.5cm]
        \item First, we will recall the well-known description of abstract
deformations of $V$ as a scheme. As we will see, $V$ has a property that we will
call \emph{essential rigidity\/}.
        \item Using this essential rigidity, we will show that any deformation
of $V$ as a closed subscheme of $(\bP^2)^n$ arises from a deformation of
$\bPhi$. In the collinear case this is non-trivial, because
$\Res(\Phi)\to(\bP^2)^n$ contracts a line, but a simple argument with the cotangent complex gives the desired result.
        \item Using \cref{prop:iso-iso}, we have that two
deformations of $\bPhi$ give rise to the same deformation of $V$ if
and only if they are isomorphic, completing the proof.
\end{enumerate}

It is worth noting (as hinted at in this outline) that the proof we give here is
almost purely geometric. We do not rely on dimension estimates, ideal-theoretic
calculations, etc. The arguments are simple variants of classical Italian
geometric arguments, first used to study the geometry of projective surfaces.
\Cref{prop:defo-proj} is ultimately the reason that the space of multiview
configurations admits an open immersion into the Hilbert scheme, as we will see
in \cref{sec:morph-hilb}.

\subsubsection{Essential rigidity of blowups of $\bP^3$}
\label{sec:rigid}

In this section we fix a commutative ring $A_0$, a square-zero
extension $$I\subset A\to A_0,$$ and a collection of pairwise
everywhere-disjoint sections $$\sigma_i:\Spec A_0\to\bP^3_{A_0}.$$ We
write $P_0$ for the blowup $\Bl_{Z_0}\bP^3_{A_0}$, where $Z_0$ is the
reduced closed subscheme of $\bP^3_{A_0}$ supported on the union of
the images of the $\sigma_i$. For the most part, these results are well-known. 
Unfortunately, the available literature tends not work in sufficient generality 
(for example, \cite{sernesi2007deformations} works over a fixed field 
$\mathbf{k}$).

\begin{prop}\label{sec:essent-rigid-blow-1} Given a deformation $P$ of
$P_0$ over $A$, there is a unique morphism $$\beta:P\to\bP^3_A$$
deforming the canonical blow-down map $$\beta_0:P_0\to\bP^3_{A_0},$$
up to infinitesimal automorphism of $\bP^3_{A}$. Moreover, $\beta$
realizes $P$ as the blowup of $\bP^3_A$ at a closed subscheme $Z$ that
deforms $Z_0$ (and $Z$ is a union of $n$ sections of $\bP^3_A$).
\end{prop}
\begin{proof} 
  If one is willing to work entirely over a field (although we are here working
  over $\Z$), one can extract this from \cite[Proposition
  3.4.25(ii)]{sernesi2007deformations}. It is not difficult to prove this in
  full generality for blowups of projective spaces along collections of
  sections, by showing that the blowdown map admits a canonical deformation, and
  each deformed exceptional divisor maps to a section under this deformed
  blowdown. We omit the details for the sake of space.  
\end{proof}

\subsubsection{Lifting deformation for non-collinear configurations}
\label{sec:lift}

In this section, we explain how any deformation of a non-collinear
multiview scheme lifts to a deformation of the associated multiview
configuration. Fix a deformation situation $$I\subset A\to A_0$$ and a
non-collinear multiview configuration $\bPhi^0$ of length $n$ over
$A_0$ with scheme $\Scheme(\bPhi^0)$.

\begin{prop}\label{sec:lift-deform-mult-1} If $X\subset(\bP^2)^n_A$ is
an $A$-flat deformation of $\Scheme(\bPhi^0)$ then there is a
deformation $\bPhi$ of $\bPhi^0$ such that $\Scheme(\bPhi)=X$ as
closed subschemes of $(\bP^2)^n$. Moreover, $\bPhi$ is unique up to
unique isomorphism of deformations of $\bPhi^0$ over $A$.
\end{prop}
\begin{proof} Since $\bPhi^0$ is non-collinear, the natural morphism
        $$\Res(\bPhi^0)\to\Scheme(\bPhi^0)\subset(\bP^2)^n$$ is an isomorphism.
        By \cref{sec:essent-rigid-blow-1}, any deformation
of $\Scheme(\bPhi^0)$ is a blowup $P$ of $\bP^3_A$ at $n$ disjoint
sections over $\Spec A$. The deformation thus results in a rational
map
        $$\bPhi:\bP^3_A\dashrightarrow(\bP^2_A)^n$$
        extending $\bPhi^0$. We wish to show that $\bPhi$ is a
relative multiview configuration in the sense of 
\cref{defn:general-rel}. To do this, it suffices to check that
composition with each projection is a relative pinhole camera. Write
$p:\bP^3_A\dashrightarrow\bP^2_A$ for one such projection; we will
abuse notation and also write $p$ for the corresponding map
$P\to\bP^2_A$ from the blowup. We will write $E$ for the exceptional
divisor associated to $p$ and $Z$ for the section blown up to make
$E$. That is, we assume that $p$ is the $i$th projection of $\bPhi$
and that $E$ is the preimage of the $i$th section in $\bP^3_A$, which
we call $Z$, uniformly omitting $i$ from the notation. By the pinhole
camera assumptions on $\bPhi^0$, $p|{E_{A_0}}$ maps $E$ isomorphically
to $\bP^2_{A_0}$. It follows from Nakayama's lemma that $p|_E$ maps
$E$ isomorphically to $\bP^2_A$.
        
        Write $U\subset\bP^3_A$ for the complement of the sections
that are blown up to resolve $\bPhi$. By the previous paragraph, we
see that $U_{A_0}\subset\bP^3_{A_0}$ is precisely the complement of
the camera centers of $\bPhi^0$. By the universal property of
projective space, the morphism $p$ is given by a surjective
morphism $$\lambda:\ms O_P^{\oplus 3}\to\ms L$$ for some $\ms L$ in
$\Pic(P)$. Write $\pi:P\to\bP^3_A$ for the blow-down map. We know from
the definition of pinhole cameras, the rigidity of invertible sheaves
on $P$, and the canonical way to extend morphisms generically across
blowups that $\ms L\cong\pi^\ast(\ms O(1))(-E)$. Moreover, the
resulting arrow
\[
  f:\pi_\ast\ms O^{\oplus 3}\to\ms O_{\bP^3_A}(1)
\]
has the property that its image is precisely $\ms O_{\bP^3_A}(1)\tensor \ms 
I_Z,$ where
$\ms I_Z$ is the ideal sheaf of $Z$. (This follows from the
universal property of blowing up.) This shows that the cokernel of $f$
is an invertible sheaf supported on $Z$, showing that $p$ is a
relative pinhole camera, as desired.
        
        It remains to show that any two such realizations $\bPhi_1$
and $\bPhi_2$ are conjugate by an infinitesimal automorphism of
$\bP^3$. But this follows immediately from 
\cref{prop:iso-iso}.
\end{proof}

\subsubsection{Lifting deformations for collinear configurations}
\label{sec:collinear-def-lift} 

For the sake of computational ease, in this section we consider a
deformation situation $I\subset A\to A_0$ in which $A$ is an Artinian
local ring with maximal ideal $\mf m$ and $\mf mI=0$. Write $k=A/\mf
m$.

We start with a multiview configuration
$\bPhi:\bP^3_{A_0}\dashrightarrow(\bP^2)^n$ whose special fiber
$\bPhi_k$ is collinear. Thus, the morphism
$$\Res(\bPhi_k)\to\Scheme(\bPhi_k)\subset(\bP^2)^n$$
contracts a line $\ell\subset\Res(\bPhi_k)$. To make things easier to
read, write $R=\Res(\bPhi_k)$ and $B=\Scheme(\bPhi_k)$. Write
$L_{R/B}$ for the cotangent complex of the morphism $R\to B$. In
addition, write $E_1,\ldots,E_n\subset R$ for the exceptional
divisors. The usual calculations show that $K_R=\pi^\ast
K_{\bP^3}+2E_1+\cdots+2E_n$.

\begin{lem}\label{lem:cotangent-vanishing} If $n>2$ then 
$\Ext^2_{R}(L_{R/B}, \ms O_R)=0$.
\end{lem}
\begin{proof} Consider the standard spectral sequence
        \begin{equation}\label{eq:E2} E_2^{pq}=\Ext^p(\ms
H^{-q}(L_{R/B}, \ms O_R))\Rightarrow\Ext^{p+q}(L_{R/B},\ms O_R).
        \end{equation} We know that $\ms H^0(L_{R/B})=\Omega^1_{R/B}$,
and that $\ms H^{-j}(L_{R/B})$ is supported on $\ell$ for all $j\geq
0$. By Serre duality, we can compute the terms in the spectral
sequence as
        $$\Ext^p(\ms H^{-q}(L_{R/B}),\ms O_R)=\H^{3-p}(R,\ms
        H^{-q}(L_{R/B})(K_R))^\vee.$$ Since the cohomology sheaves of
$L_{R/B}$ are all supported on $\ell$, all columns of the $E^2_{pq}$
page \eqref{eq:E2} vanish except (possibly) for $p=2,3$. It follows
that
        $$\Ext^2_{R}(L_{R/B},\ms O_R)\cong\H^1(R,\Omega^1_{R/B}(K_R))^\vee.$$
        
        A local calculation shows that $\Omega^1_{R/B}$ is annihilated
by the ideal of $\ell$, so that $\Omega^1_{R/B}=\Omega^1_{\ell/\Spec
k}$, and thus
        $$\H^1(R,\Omega^1_{R/B}(K_R))^\vee\cong\H^1(\ell, \ms
        O_{\ell}(K_\ell+K_R))^\vee\cong\H^0(\ell, \ms
O_\ell(-K_R))=\H^0(\ell, \ms O(4-2n))=0,$$ as desired.
\end{proof}

\begin{prop}\label{prop:def-coll} Suppose $n>2$. If
$X\subset(\bP^2)^n_A$ is an $A$-flat deformation of $\Scheme(\bPhi^0)$
then there is a deformation $\bPhi$ of $\bPhi^0$ such that
$\Scheme(\bPhi)=X$ as closed subschemes of $(\bP^2)^n$. Moreover,
$\bPhi$ is unique up to unique isomorphism of deformations of
$\bPhi^0$ over $A$.
\end{prop}
\begin{proof} By \cref{lem:cotangent-vanishing} and \cite[ 
III.2.2.4]{MR0491680}, the obstruction
to deforming the morphism $$\Res(\Phi^0)\to\Scheme(\Phi^0)$$ over $A$ vanishes,
resulting in a deformation $R\to X$. Applying the results of 
\cref{sec:rigid}, we see that this arises from a deformation $\Phi$, as
desired. The uniqueness of $\Phi$ up to isomorphism is an immediate
consequence of \cref{prop:iso-iso}.
\end{proof}

\subsection{Diagram Hilbert schemes}
\label{sec:hilbert}

In this section, we briefly explain a basic idea that is hard to find in the 
literature: diagram $\Hom$-schemes and diagram Hilbert schemes. They are a mild 
elaboration of the idea of a flag Hilbert scheme. By not only remembering the 
data of the image but also the calibrating conics the moduli of calibrated 
cameras maps to a diagram Hilbert schemes in the same way that the moduli of 
uncalibrated cameras maps to a Hilbert scheme. 

\subsubsection{Definition and examples}
\label{sec:definition-examples}

Fix a base scheme $S$, a category $I$, and a functor 
$\underline{X}:I\to\AlgSp_S$, where $\AlgSp_S$ denotes the category of algebraic 
spaces over $S$.

\begin{defn}\label{defn:diag-hilb}
  The \emph{diagram Hilbert functor\/} $$\Hilb_{\underline 
X}:\Sch_S^\circ\to\Sets$$ is
  the functor whose value on an $S$-scheme $T$ is the set of
  isomorphism classes of natural transformations $\underline
  Y\to\underline X\times_S T$ of functors $I\to\Sch_T$ where for each
  $i\in I$ the associated arrow $\underline Y(i)\to\underline
  X(i)\times_S T$ is a $T$-flat family of proper closed subschemes of
  $\underline X(i)$ of finite presentation over $T$.
\end{defn}

\begin{example}
  The usual Hilbert scheme is an example: just take $I$ to be the
  singleton category. So is the flag Hilbert scheme of length $n$: in
  this case the category $I$ is the category $\underline n$ associated
  to the poset $\{1,\ldots,n\}$, and the functor $\underline X$ is the
  constant functor $X\to X$. A natural transformation $\underline
  Y\to\underline X$ defines a nested sequence of closed subschemes of
  $X$. This is the flag Hilbert scheme (of length 2 flags).

 There is also a stricter kind of flag scheme: suppose $X_1\subset
 X_2$ is a closed immersion and one wants to parameterize pairs
 $Y_i\subset X_i$ such that $Y_1\subset Y_2$. That is precisely the
 diagram Hilbert functor associated to the poset-category $\underline 2=\{1<2\}$
 with the functor $\underline 2\to\Sch_S$ sending $i$ to $X_i$. This
 last example is the one that will arise naturally for us in the
 context of calibrated cameras. (We record more general results here
 in case someone in the future needs this general idea of diagram
 Hilbert scheme.)
\end{example}

\begin{notn}
  If the diagram in question is a single morphism $X\to Y$, we will
  write
$\Hilb_{X\to Y}$ for the associated Hilbert functor.
\end{notn}

\subsubsection{Representability}
\label{sec:representability}

The main result about diagram Hilbert functors is that they are
representable. We prove this in a high degree of generality, in case this is of independent interest.

\begin{prop}\label{P:diag-hilb-repr}
  Let $I$ be a finite category and $\underline X:I\to\AlgSp_S$ a functor
  whose components are separated algebraic spaces. Then the diagram Hilbert
  functor $\Hilb_{\underline X}$ is representable by an algebraic
  space locally of finite presentation over $S$. If the $\underline
  X(i)$ are locally quasi-projective schemes then $\Hilb_{\underline
    X}$ is represented by a locally quasi-projective $S$-scheme.
\end{prop}
\begin{proof}
  There is a natural functor
$$F:\Hilb_{\underline X}\to\prod_{i\in I}\Hilb_{\underline X(i)},$$
and we know that the latter is representable by algebraic spaces
(resp.\ schemes) satisfying the desired conditions. It thus suffices
to show the same for $F$, i.e., that $F$ is representable by spaces of
the required type.

For each $i\in I$, let $$Z_i\subset\underline
X(i)\times\prod\Hilb_{\underline X(i)}$$ denote the universal closed
subscheme (pulled back over the product). Let $A$ denote the set of
arrows in $I$; for an arrow $a\in A$, let $s(a)$ and $t(a)$ denote the
source and target of $a$. Consider the scheme
$$H:=\prod_{a\in A}\Hom_{\Hilb_{\prod\underline X(i)}}(Z(s(a)),
Z(t(a))),$$
which naturally fibers over $\prod\Hilb_{\underline X(i)}$. The
standard theory of $\Hom$-schemes shows that
$H\to\prod\Hilb_{\underline X(i)}$ is representable by spaces of the
desired type.

The final observation to make is that composition of two arrows gives
equations $b\circ a=c$ in $A$, and these translate into
\emph{closed\/} conditions on $H$ because all of the subschemes $Z(i)$
are separated. Since the conditions desired are stable under taking
closed subspaces, we have proven the result.
\end{proof}

\subsection{Moduli of calibrated camera configurations}
\label{sec:mod-cal}

Let $\ms C$ denote the space of smooth conics in
$\bP^2_{\Spec\Z[1/2]}$, and let $C_\univ\subset\bP^2_{\ms C}$ denote the
universal smooth conic. (The space $\ms C$ is an open subscheme of the
bundle of sections of $\ms O_{\bP^2_{\Spec\Z[1/2]}}(2)$.) The tuple of
conics $(C_\univ,\ldots,C_\univ)$ inside $(\bP^2)^n$ will be called
the \emph{universal calibration\/}.

\begin{defn}
  Given a positive integer $n$, the \emph{sheaf of calibrated camera
    configurations of length $n$\/}, denoted $\CalCam_n$, is the sheaf over the
    Cartesian power $\ms C^n$ whose value over a point $t:S\to\ms C^n$ consists
    of the set of isomorphism classes of general relative calibrated multiview
    configurations of length $n$ with calibration datum of the form
    $(C,t^\ast(C_\univ,\ldots,C_\univ))$.
\end{defn}

In down-to-earth terms, we are just describing the space of $n$-tuples
of calibrated cameras with pairwise non-intersecting centers, together
with \emph{arbitrary but specified\/} calibration data. In the
existing literature, the word ``calibrated'' usually means that one
has fixed the calibrating conics to be the canonical absolute conic in space
(attached to the Euclidean distance form on $\bP^3$) and the circle in
the plane. Since any two smooth conics are conjugate under a
homography, this seems harmless. As we hope to describe in this
section, thinking more geometrically and \emph{tracking the conics as
data instead of normalizing them\/} gives us a great deal of insight
into the underlying moduli problem. The point of the universal conic
in $\bP^2$ is that we only want to allow the conic in $\bP^3$ to vary;
that is, we fix calibration data on the image planes when we
define the moduli problem. By working with the universal conic, we
allow those fixed planar data to be arbitrary.

\begin{notn}
  Since we are fixing the calibration data on the image planes to be
  the universal conic, we will omit them from the notation for a
  calibration datum. Thus, we will write $(\bPhi,C)$ for a calibrated
  configuration. When we need to refer to the image plane calibrating
  curves, we will use $C_i$ for the curve in the $i$th plane, it is
  key to remember that while $C_i$ can vary as the base varies
  (depending upon how it maps to $\ms C^n$), this is determined solely
  by the base and not by the object of $\CalCam_n$ over that point of
  the base.
\end{notn}

The main result of this section is the following.
\begin{prop}\label{P:cal-cam-yummy}
  The sheaf $\CalCam_n$ is a smooth scheme of finite type
  over $\ms C^n$.
\end{prop}

Let $\tau_n : \CalCam_n \to \CalCam_{n-1}\times_{\ms C^{n-1}}\ms C^n$
be the morphism given by forgetting the last camera (and retaining the
last calibrating plane conic).

\begin{lem}\label{L:calcam-forget-repble}
  The morphism $\tau_n$ is representable by separated schemes of
  finite presentation. 
\end{lem}
\begin{proof}
  Let $((\phi_1,\ldots,\phi_{n-1},C), C_n)$ be a $T$-valued point of
  $\CalCam_{n-1}\times_{\ms C^{n-1}}\ms C^n$. The fiber of $\tau_n$
  is given by the set of cameras $\phi_n$ with the same domain $\bP\to T$ as the
  first $n-1$ cameras, with the following additonal properties.
  \begin{enumerate}[leftmargin=1.5cm]
  \item The center of $\phi_n$ avoids the centers of $\phi_i$ for
    $i=1,\ldots,n-1$.
  \item The restriction $\phi_n|C$ factors through the closed
    subscheme $C_n\subset\bP$.
  \end{enumerate}
  The space of camera centers satisfying the first condition is an
  open subscheme $\bP^\circ\subset\bP$, and taking the center gives a
  natural map $$\CalCam_n\to\bP^\circ\times\CalCam_{n-1}\times_{\ms
    C^{n-1}}\ms C^n.$$ It
  suffices to show that this map is representable, and thus we may
  assume that the
  center is a given section $\sigma:T\to\bP$. Blowing up along $\sigma(T)$ to 
yield
  $\widetilde{\bP}$, with exceptional divisor $E$, we can then realize
  the cameras inside the open locus of the 
  $\Hom$-scheme $\Hom(\widetilde{\bP},\bP^2)$ parametrizing maps 
$f:\widetilde{\bP}\to\bP^2$ for which $f^\ast\ms
  O_{\bP^2}(1)$ is isomorphic to $\ms O(1)(-E)$ on each geometric
  fiber over $T$. This locus is of
  finite type. Finally, the condition that $C$ lands in $C_n$ is
  closed (and of finite presentation), completing the proof.
\end{proof}

\begin{prop}\label{sec:moduli-calibr-camera}
  The morphism $\tau_n$ is smooth.
\end{prop}

\begin{proof}
  By \cref{L:calcam-forget-repble} and \cite[Tag
  02H6]{stacks-project}, it suffices to show that $\tau_n$ is formally
  smooth. Let $A \to A_0$ be a square-zero extension of rings, and
  suppose that $$(\phi_1, \dots, \phi_n, C) \in \CalCam_n(A_0)$$ is
  fixed. To show formal smoothness we can work Zariski-locally and
  thus assume that the domains of $\phi_1, \dots, \phi_n$ are
  $\bP^3_{A_0}$. Now suppose that we fix a deformation $$((\phi'_1,
  \dots, \phi'_{n-1}, C_A), C_n) \in \CalCam_{n-1}(A)\times_{\ms C^{n-1}(A)}\ms
  C^n(A).$$
  (Because we are working over the universal conic in each image
  plane, we have to specify the deformation of the conic that we will
  use in attempting to deform the $n$th calibrated camera.) 
  To show formal smoothness
  is suffices to extend $\phi_n$ to a morphism $\phi'_n$ that maps
  $C_A$ to $C_n$.   
  
  The choice of deformation of $C$ to $C_A$ induces a lift of $C \to
  \bP^2_{A_0}$ to $C_A \to \bP^2_A$. This is because $\H^1(C,\ms O_C(1))=0$, so sections defining a map can always be lifted. We will show that we can extend this to a camera that acts on $C_A$ in the given way.
  
  We are thus reduced to the following: we are given a tuple of three sections 
$\sigma_0,\sigma_1,\sigma_2\in\Gamma(\bP^3_{A_0}, \ms O(1))$, a planar curve 
$C_A\subset \bP^3_A$ of degree $2$, and lifts of the $\sigma_j|_{C}$ to $\Gamma(C_A,\ms 
O(1))$. We wish to lift these extensions to sections 
$\widetilde\sigma_j\in\Gamma(\bP^3_A,\ms O(1))$. We can do this one section at a 
time. By \cref{defn:calibration-datum}, the curve $C_A$ is contained in a canonically 
defined family of planes in $\bP^3_A$; we will write  $C_A\subset\bP^2_A\subset\bP^3_A$ and 
similarly for $A_0$. (If the plane is not trivial, we can further shrink $A$ to 
make it so; this is immaterial for the calculations and is only a notational 
device.)

  Consider the diagrams 
  \[
    \begin{tikzcd}
      0\ar[r] & \Gamma(\bP^3_{A_0},\ms O)\tensor_{A_0}I\ar[r]\ar[d] &
      \Gamma(\bP^3_A,\ms O)\ar[r]\ar[d] & \Gamma(\bP^3_{A_0},\ms O)\ar[r]\ar[d] 
& 0\\
      0\ar[r] & \Gamma(\bP^3_{A_0},\ms O(1))\tensor_{A_0}I\ar[r]\ar[d] &
      \Gamma(\bP^3_A,\ms O(1))\ar[r]\ar[d] & \Gamma(\bP^3_{A_0},\ms 
O(1))\ar[r]\ar[d] & 0\\
      0\ar[r] & \Gamma(\bP^2_{A_0},\ms O(1))\tensor_{A_0}I\ar[r] &
      \Gamma(\bP^2_A,\ms O(1))\ar[r] & \Gamma(\bP^2_{A_0},\ms O(1))\ar[r] & 0
    \end{tikzcd}
  \]
and
  \[
     \begin{tikzcd}
      0\ar[r] & \Gamma(\bP^2_{A_0},\ms O(-1))\tensor_{A_0}I\ar[r]\ar[d] &
      \Gamma(\bP^2_A,\ms O(-1))\ar[r]\ar[d] & \Gamma(\bP^2_{A_0},\ms 
O(-1))\ar[r]\ar[d] & 0\\
      0\ar[r] & \Gamma(\bP^2_{A_0},\ms O(1))\tensor_{A_0}I\ar[r]\ar[d] &
      \Gamma(\bP^2_A,\ms O(1))\ar[r]\ar[d] & \Gamma(\bP^2_{A_0},\ms 
O(1))\ar[r]\ar[d] & 0\\
      0\ar[r] & \Gamma(C,\ms O(1))\tensor_{A_0}I\ar[r] &
      \Gamma(C_A,\ms O(1))\ar[r] & \Gamma(C,\ms O(1))\ar[r] & 0.
    \end{tikzcd}
  \]
  By the usual calculations of the cohomology of projective space, these two
  diagrams have exact columns. A simple diagram chase then shows that we can
  lift sections to $\bP^3_A$ given values on $\bP^3_{A_0}$ and $C_A$, completing
  the proof.
\end{proof}

\begin{proof}[Proof of \cref{P:cal-cam-yummy}] It remains to show
  smoothness. We use \cref{sec:moduli-calibr-camera} and induction on
  $n$. For $n=1$, we see that $\CalCam_1$ is smooth over $\ms C$, which is
  itself open in a projective space, hence smooth.
\end{proof}

\subsection{Deformation theory of calibrated camera configurations}
\label{sec:def-cal}

In this section we prove the following analogue of 
\cref{prop:defo-proj}.

\begin{thm}\label{thm:cal-defos-same} 
If $(\bPhi,C)$ is a non-degenerate calibrated 
general multiview configuration of length $n>2$ with associated
multiview flag
$$(C\subset V)\hookrightarrow(C_1\times\cdots\cdot
C_n\subset(\bP^2)^n)$$ then we have that the infinitesimal
deformations of $(\bPhi,C)$ are in bijection with the infinitesimal
deformations of $C\subset V$ as a closed subscheme diagram of
$C_1\times\cdots\times C_n\subset(\bP^2)^n$.
\end{thm}
\begin{proof}
  The proof leverages the proof of \cref{prop:defo-proj}. In
  particular, we can forget the calibrations and apply 
 \cref{prop:defo-proj} to see that under the given hypotheses any
  deformation of $\Flag(\bPhi,C)$ induces a deformation of
  $\Scheme(\bPhi)$ that is the image of a deformation $\widetilde\bPhi$ of 
$\bPhi$. The
  assumption that the deformation of $\Scheme(\bPhi)$ arises from a
  deformation of $\Flag(\bPhi,C)$ means that there is also an
  associated deformation of $C$. Since $\bPhi$ is an isomorphism onto
  its image in a neighborhood of $C$, this deformation of $C$
  canonically lifts to give a calibration of $\widetilde\bPhi$.
\end{proof}

\section{Comparison morphisms}
\label{sec:comparisons}
In \cref{sec:forget-map} we compare $\Cam_n$ and $\CalCam_n$ by the 
natural decalibration morphism. In \cref{sec:twisted-pairs-intro} we 
focus on the case of two cameras, leading to a 2-1 cover of the esesntial 
variety that compactifies the twisted pair covering. Finally, in \cref{sec:morph-hilb} we state how both moduli spaces of cameras map to 
appropriate Hilbert schemes. 

\subsection{The decalibration morphism $\nu_n:\CalCam_n\to\Cam_n\times\ms C^n$}
\label{sec:forget-map}

In this section, we study a natural morphism
$$\CalCam_n\to\Cam_n\times\ms C^n$$
given by forgetting the camera calibration datum.  

\begin{defn}
  The \emph{decalibration morphism\/} is the morphism
$$\nu_n:\CalCam_n\to\Cam_n\times\ms C^n$$
given by sending $(\bPhi,C)$ to $\bPhi$. 
\end{defn}


\subsubsection{Intersections of conic cones}
\label{sec:cone-int}

Before we delve into the geometry of $\nu_n$, we need a few preliminaries about
intersections of conic cones in $\bP^3$.

\begin{prop}\label{P:cone-int}
  Let $X_1$ and $X_2$ be two conic cones in $\bP^3$ with distinct cone
  points $P_1$ and $P_2$. Suppose $C\subset X_1\cap X_2$ is a plane
  curve of degree $2$, so that $X_1\cap X_2=C\cup D$ with $D$ a curve
  of degree $2$. Then $D$ must be planar and have support distinct
  from the support of $C$. More precisely, one of the following must
  occur.
  \begin{enumerate}[leftmargin=1.5cm]
   \item $C$ and $D$ are smooth conics meeting at two distinct points.
   \item $C$ is a smooth conic and $D$ is a doubled planar line.
   \item $C$ is a doubled planar line and $D$ is a smooth conic.
  \end{enumerate}
  In particular, we can never have $C=D$ (i.e., $X_1\cap X_2$ cannot
  be doubled smooth conic).
\end{prop}
\begin{proof}
  This is a standard result, and it can be extracted from the material in
  \cite[Chapter 13, Section 11]{MR1288306}. We briefly describe a proof in
  modern language for the convenience of the reader. By assumption, $C$ is
  either a smooth conic or a planar doubled line. It is easy to write down
  examples where the intersection $X_1\cap X_2$ is a union of two smooth conics
  meeting at two points (e.g., in characteristic different from $2$ the pair
  $X^2+Y^2+Z^2=0$ and $Y^2+Z^2+W^2=0$ is such an example).
    
  If $X_1\cap X_2$ contains a doubled planar line, then $X_1$ and
  $X_2$ must be tangent along a ruling. Since $P_1\neq P_2$, the
  residual curve must be a smooth conic.
    
  Suppose $X_1\cap X_2=C\cup D$ with $C$ a smooth conic and $D$ a
  singular curve. We wish to show that $D$ is a doubled planar
  line. Since $D$ has degree $2$ in $\bP^3$, it must be the case that
  the reduced structure on $D$ is a line. The only doubled lines
  contained in a conic cone are planar: they are given by intersecting
  with the tangent plane along rulings.

  It remains to rule out the possibility that $X_1\cap X_2$ is a doubled conic. Note that a doubled conic is
  the intersection of $X_1$ with a doubled plane $2P\in\ms
  O_{\bP^3}(2)$. We can rule out this case if we can show that the
  pencil spanned by $X_1$ and a doubled plane not containing its cone
  point does not contain any more conic cones. We can represent the cone $X_1$ and an aribtrary doubled plane missing the cone point by
  the matrices 
$$\begin{pmatrix}
1 & 0 & 0 & 0\\
0 & 1 & 0 & 0\\
0 & 0 & 1 & 0\\
0 & 0 & 0 & 0
\end{pmatrix}\text{\rm\ and } \begin{pmatrix}
a^2 & ab & ac & a\\
ab & b^2 & bc & b\\
ac & bc & c^2 & c\\
a & b & c & 1
\end{pmatrix} $$
for $a,b,c\in k$. Searching for a conic cone in the pencil corresponds
to finding $\lambda$ such that the following matrix has rank 3:
$$\begin{pmatrix}
a^2+\lambda & ab & ac & a\\
ab & b^2+\lambda & bc & b\\
ac & bc & c^2+\lambda & c\\
a & b & c & 1
\end{pmatrix} \text{\rm\ with row reduction } 
\begin{pmatrix}
\lambda & 0 & 0 & 0\\
0 & \lambda & 0 &0 \\
0 & 0 & \lambda & 0\\
a & b & c & 1
\end{pmatrix}.$$
But the latter matrix can never have rank $3$.
\end{proof}

\subsubsection{The geometry of $\nu_n$}
\label{sec:geom-map}

Fix a point $\xi$ of $\Cam_n\times\ms C^n$. That is, fix conics
$C_1,\ldots,C_n$ in $\bP^2$ and a multiview configuration $\bPhi$. In
this section we compute the fiber of $\nu_n$ over $\xi$.

\begin{prop}\label{P:tiny-fibers}
  The scheme-theoretic fiber $\nu_n^{-1}(\xi)$ is a reduced
  $\kappa(\xi)$-scheme of length at most $2$.
\end{prop}
\begin{proof}
  The fiber $\nu_n^{-1}(\xi)$ is precisely the scheme of smooth conics
  in the intersection of the cones over the image conics $C_i$ inside
  the ambient $\bP^3$. The result is thus immediate from 
 \cref{P:cone-int}. (In particular, the lack of doubled conic means
  that the fibers are discrete.) 
\end{proof}

\begin{cor}\label{C:unram-nu}
  The morphism $\nu_n$ is unramified.
\end{cor}
\begin{proof}
  This is an immediate consequence of \cref{P:tiny-fibers}.
\end{proof}

\begin{prop}\label{L:nu_n-closed-image}
  The morphism $\nu_n$ is proper.
\end{prop}

\begin{proof}
  Suppose we have a multiview configuration $\bPhi$ of length $2$ over
  a complete dvr $R$ with fraction field $K$, degree two curves
  $C_1, \dots, C_n \subset \bP_R^2$ and a degree two curve
  $C_K\subset\bP^3_K$ such that $\bPhi_K$ maps $C_K$ isomorphically to
  the generic fiber of each $C_i$. By the valuative criterion for
  properness it suffices to extend $C_K$ to a degree two curve $C_R$.
  
  Assume we have a multiview configuration $\bPhi$ of length $2$ over
  a complete dvr $R$ with fraction field $K$, and suppose we have
  conics $C_1, \dots, C_n \subset \bP_R^2$ in each image plane. Write
  $\widebar C_i \subset \bP^3$ for the cone over $C_i$ under
  $\pr_i \circ \bPhi$ and
  $I=\widebar C_1 \cap \cdots \cap \widebar C_n$. Finally, assume that
  there is a conic $C_K\subset\bP^3_K$ such that $\bPhi_K$ maps $C_K$
  isomorphically to the generic fiber of each $C_i$; that is,
  $C_K \subset I_K$. Let $C_R$ be the specialization of $C_K$ in the
  closed fiber $C_0$. The curve $C_R$ is degree $2$, giving us a
  calibrated configuration over $R$.
\end{proof}

Note that even if $C_k$ is a non-degenerate conic, $C_0$ need not
be. This is why we need to add degenerate conics.

\begin{prop}\label{P:nu_n-smooth}
  The morphism $\nu_2$ has smooth image and general fiber of length
$2$. For any $n>2$ the morphism $\nu_n$ is generically injective.
\end{prop}
\begin{proof}
  The projective closure of the image of a fiber of $\CalCam_2$ over $\ms C^2$
  under $\nu_2$ is known as the ``essential variety'', and its singularities are
  well-known (see \cite[\S 2.1]{kileel}); none of its singular points lies in
  the image of $\nu_2$. To study the general fiber, it suffices by the
  irreducibility of all spaces involved to produce a single example of a camera
  configuration of length two such that the fiber of $\nu_2$ has length $2$. To
  do this, it further suffices to find a single example of two conic cones
  $C_1,C_2\subset\bP^3$ whose intersection is a pair of smooth conics. One
  such example is given by the cones $X^2+Y^2+Z^2=0$ and $Y^2+Z^2+W^2=0$.

  We now show that $\nu_n$ is generically injective for $n>2$. Given a smooth
  conic $C$ in $\bP^3$, the locus in $|\ms O_{\bP^3}(2)|$ consisting of conic
  cones containing $C$ is $3$-dimensional (since such a cone is determined by
  its vertex). Thus, we can find three non-collinear conic cones that contain
  any given smooth conic $C$. On the other hand, given two conic cones
  $C_1,C_2$, the set of conic cones that vanish on their entire intersection
  $C_1\cap C_2$ is contained in the pencil spanned by $C_1$ and $C_2$. We
  conclude that if $C_1\cap C_2$ is reducible, then we can choose general cones
  $C_3,\ldots,C_n$ containing a smooth conic in $C_1\cap C_2$ such that $C_i$ is
  not in the pencil spanned by $C_1$ and $C_2$ for each $i>2$. The joint
  vanishing locus $C_1\cap C_2\cap C_3\cap\cdots\cap C_n$ is a smooth conic.
  Since this is generic behavior, this shows that $\nu_n$ is generically
  injective for all $n>2$.
\end{proof}

It is a potentially interesting problem to characterize the locus over which
$\nu_n$ is not injective, and the singular locus of its image (the
``variety of calibrated $n$-focal tensors'', which is studied for
$n=3$ in coordinatized form in \cite{kileel2017minimal}).

\begin{cor}\label{cor:decal-finite}
  The morphism $\nu_n$ is finite.
\end{cor}

\begin{proof}
  We have shown that $\nu_n$ is quasi-finite and proper and thus, finite.
\end{proof}


\begin{question}
Is the singular locus of the image of $\CalCam_n$, for $n>2$, equal to the locus over which the fiber of $\nu_n$ has length $2$?
\end{question}

\subsection{Twisted pairs and moduli}
\label{sec:twisted-pairs-intro}

In this section we study the morphism $\nu_2$ in more detail, showing how the
Hilbert scheme gives a natural compactification of the classical ``twisted
pair'' construction. To explicitly compare this new treatment with the
literature, in this section we will fix the calibrating conics to be
$v(x_0^2+x_1^2+x_2^2) \subset \bP^2$. Also, we will often think of an essential
matrix as the corresponding pair of calibrated cameras in normalized
coordinates. In these coordinates we can fix notation $P_1 = [I|0]$ and $P_2 =
R[I|t]$ where $t=(a,b,c)$.

\subsubsection{Twisted pairs}

As shown in  5.2 of \cite{MR1226451}, the locus $\mathcal M$ of essential
matrices is smooth (over $\C$) and admits an \'{e}tale surjection
$\SO(3)\times\bP^2\to\mathcal M$, coming from composing a camera with a rotation
and a translation, up to scaling. In terms of matrices we send $(R,t)$ to the
camera pair $P=[I| 0], Q=[R|t]$ which has essential matrix $[t]_\times R$. One
can check in local coordinates that the map is \'etale
\cite[3.2]{demazure:inria-00075672}.

For any \emph{real\/} essential matrix $M\in\mathcal M(\R)$, the fiber
of $\pi$ over $M$ contains two points: one can take a pair of cameras
$P_1,P_2$ and replace it with the pair $P_1,\widetilde P_2$ where
$\widetilde P_2$ results from rotating $P_2$ by 180 degrees around the
axis connecting the centers of $P_1$ and $P_2$. In normalized
coordinates, the matrix
  \[
    R_t = 
    \begin{pmatrix}
      2a^2-1 & 2ab & 2ac & 0 \\
      2ab & 2b^2-1 & 2b c & 0 \\
      2ac & 2bc & 2c^2-1 & 0 \\
      0 & 0 & 0 & 1
    \end{pmatrix}
  \]
  is rotation by 180 degrees and $\widetilde P_2 = R[I|t] R_t $. (Note
  that over the reals we can always rescale $t$ so that
  $a^2+b^2+c^2=1$. ) 
  The pair $(P_1,P_2), (P_1,\widetilde P_2)$ is
  called a \emph{twisted pair\/}; what we have described is a
  well-known construction in computer vision \cite[Result 
9.19]{hartley2003multiple}. The key thing to note is that the rotation
  construction described above \emph{preserves calibrations\/} for
  real cameras. For complex cameras, things get more complicated, and
  for displacements $(a,b,c)$ such that $a^2+b^2+c^2=0$, the
  corresponding transformation produces a new camera pair
  $(P_1,\widetilde P_2)$ for which $\widetilde P_2$ is no longer
  calibrated.

\subsubsection{Compactification of the twisted pair construction}

The morphism $$\nu_2: \CalCam_2 \to \Cam_2 \times \ms C^2$$ gives a double
covering of a closed subscheme that generalizes the twisted pair covering of the
essential variety. A point of $\CalCam_2$ is the datum $(P_1, P_2, C)$ where
$P_1$ and $P_2$ are cameras and $C$ is a planar curve of degree $2$ contained in
the intersection of the cones defined by the preimage of $C_{\univ}$ via $P_1$
and $P_2$. \cref{P:cone-int} tells us that this intersection must contain either
another non-degenerate conic or a doubled line. In either case denote this other
degree two curve by $\widetilde C$. The general fibers of $\nu_2$ are the
triples $(P_1, P_2, C)$ and $(P_1, P_2, \widetilde C)$.

This double covering agrees with the twisted pairs covering on real points. In
normalized coordinates $\widetilde C$ is defined by the simultaneous vanishing
of
\[
  x^2+y^2+z^2=0 \text{ and } (a^2+b^2+c^2)w - 2(ax+by+cz) = 0.
\]
When $a^2+b^2+c^2=1$, as it must over $\R$ (up to scaling), one can check that 
changing coordinates on $\bP^3$ via the automorphism
\[
  H = 
  \begin{pmatrix}
    1 & 0 & 0 & 0 \\
    0 & 1 & 0 & 0 \\
    0 & 0 & 1 & 0 \\
    -2a & -2b & -2c & 1
  \end{pmatrix}
\]
sends the triple $(P_1, \widetilde P_2, C)$ to the triple $(P_1, P_2, \widetilde 
C)$.

However, over the complex numbers there exist essential matrices such
that $a^2+b^2+c^2=0$. This is exactly the condition that
$\widetilde C$ is a doubled line. In this situation the twisted pair
construction fails because the camera $\widetilde P$ no longer has a
trivial calibration. Mathematically speaking, we are really discussing
the fact that the twisted pairs morphism $\pi$, while always \'etale,
is \emph{not\/} finite. Allowing degenerate calibrations (doubled
lines) extends the twisted pair morphism $\pi$ to $\nu_2$.

\begin{prop}\label{P:calcam-involution}
  There exists a fixed-point free involution, $\chi: \CalCam_2 \to \CalCam_2$
  over $\Cam_2$ given by fixing the cameras and swapping calibrating curves.
  More precisely, $\nu_2 \circ \chi = \nu_2$.
\end{prop}
\begin{proof}
  Given a pair of cameras $\bPhi \to \bP^2 \times \bP^2$ and smooth
  conics $D_1, D_2 \subset \bP^2$, we can pull back to get two cones
  $X_1, X_2 \subset \bP^3$. Let $F= X_1 \cap X_2$. Blowing up the
  camera centers, the strict transform of these cones,
  $\tilde{X_1}, \tilde{X_2} \subset \Bl_{Z_1, Z_2} \bP^3$, are smooth
  surfaces in $\bP^3$. The intersection is a relative effective
  Cartier divisor and $\tilde{X_1} \cap \tilde{X_2} \simeq F$ since
  the cone centers are distinct.
  
  A point in $\CalCam_2$ is a pair $(\bPhi, C)$ where $C$ is a
  relative effective Cartier divisor contained in $F$. By \cite[Tag
  0B8V]{stacks-project} there exists another relative effective
  Cartier divisor $C'$ such that $C' + C = F$. Checking at a geometric
  point, \cref{P:cone-int} shows that $C'$ is a degree 2
  curve, and that no geometric point of $\CalCam_2$ is fixed by $\chi$. This 
argument is functorial and so induces the desired
  involution. Since $\chi$ only changes the calibrating conic we have 
  $\nu_2 \circ \chi = \nu_2$.
\end{proof}

\begin{thm}\label{sec:comp-twist-pair}
  The morphism $\nu_2$ factors as a finite étale morphism followed by
  a closed immersion.
\end{thm}
\begin{proof}
  By Corollary \cref{cor:decal-finite}, $\nu_2$ is a finite morphism, hence
  closed. This yields a factorization $\CalCam_2\to Z\to\Cam_2$ with the second
  arrow a closed immersion and the first scheme-theoretically surjective.
  Let $A$ be a strictly Henselian local ring and $\Spec A\to Z$ a
  morphism. The finiteness of $\nu_2$ yields a diagram
  \begin{cd}
    \spec B \rar \dar{\psi} & \CalCam_2 \dar{\nu_2} \\
    \spec A \rar & \Cam_2
  \end{cd}
  By \cite[Tag 04GH]{stacks-project}, $B$ is the product of local
  Henselian rings. By \cref{P:nu_n-smooth}, the general
  fibers of $\psi$ are length $2$, corresponding to the two possible
  calibrating conics, so $\spec B \simeq \spec B_1 \sqcup \spec
  B_2$. By Corollary \cref{C:unram-nu}, $\psi$ is
  unramified, and thus (by \cite[Tag 04GL]{stacks-project}) restricts
  to a closed embedding on each $\spec B_i$. 
  \begin{cd}
    \spec B_i \rar \ar[hook]{dr} & \spec B_1 \sqcup \spec B_2 \rar \dar{\psi}   
& \CalCam_2 \dar  \\
    & \spec A \rar & \Cam_2
  \end{cd}
  The involution described in \cref{P:calcam-involution}
  induces an isomorphism $f: \spec B_1 \to \spec B_2$. In other words
  both components map isomorphically to the image, so $\nu_2$ is
  étale over $Z$, as claimed.
\end{proof}

\subsection{Morphisms to Hilbert schemes}
\label{sec:morph-hilb}

The following describes the main result relating the moduli problems $\Cam_n$
and $\CalCam_n$ to Hilbert schemes. This gives the generalization of the results
of \cite[Theorem 6]{aholt2011hilbert}, leveraging the novel methods of this
paper to give more information about the uncalibrated case and the appropriate
result in the calibrated case.

\begin{prop}\label{prop:hilb-map}
  The associations $$\bPhi\mapsto\Scheme(\bPhi)$$
  and $$(\bPhi,C)\mapsto\Flag(\bPhi,C)$$ define 
monomorphisms
$$\Scheme:\Cam_n\to\Hilb_{(\bP^2)^n/\Spec\Z[1/2]}$$ and
$$\Flag:\CalCam_n\to\Hilb_{C_\univ^n\subset(\bP^2)^n/\ms C^n}$$
 such that
\begin{enumerate}[leftmargin=1.5cm]
\item when $n>2$, the morphism $\Scheme$ (respectively, $\Flag$) itself is an 
open immersion into
  $\Hilb_{(\bP^2)^n/\Spec\Z[1/2]}^\sm$ (respectively, 
  $\Hilb_{C_\univ^n\subset(\bP^2)^n/\ms C^n}^\sm$);
\item the arrows $\Scheme$ and $\Flag$ together with the forgetful
  maps give a commutative diagram
$$
\begin{tikzcd}
\CalCam_n\ar[d, "\nu_n"']\ar[r, "\Flag"] &
\Hilb_{C_\univ^n\subset(\bP^2_{\ms C^n})^n/\ms C^n}\ar[d] \\
\Cam_n\times_{\Spec\Z[1/2]}\ms C^n\ar[r, "\Scheme"'] & \Hilb_{(\bP_{\ms C^n}^2)^n/\ms 
C^n}.
\end{tikzcd}
$$
\end{enumerate}
In particular, every geometric fiber of $\Scheme$ over $\Spec\Z[1/2]$ is an open immersion of $\Cam_n$ into the smooth locus of a single irreducible
component of the Hilbert scheme,
and similarly for geometric fibers of $\Flag$ and components of the
diagram Hilbert scheme.
\end{prop}
\begin{proof}
  \cref{base-change-image} and 
 \cref{prop:iso-iso} show that $\Flag$ is a well-defined 
  monomorphism. Since $\CalCam_n$ is smooth over $\ms C^n$, we have
  that $\Flag$ is an open immersion in a neighborhood of 
  any point where it induces an isomorphism of deformation
  functors. 
 \cref{thm:cal-defos-same} then applies to give the two desired statements. 
\end{proof}

\section{Questions}\label{sec:questions}

In this section, we briefly discuss questions raised by this work, and suggest some directions for future investigation.

\begin{question}
  What concrete computational consequences follow from functorial methods?
\end{question}
We believe that the techniques described here may be useful for studying the
numerical properties of multiview geometry. For example, in \cite{essvar}, we
will give an explicit equation for the fiber of $\CalCam_2$ over the pair of
standard Euclidean conics, which appears as a double cover of the essential
variety extending the twisted pair construction. It is given by the vanishing of
a single bilinear form on $\bP^3\times\bP^3$. This can be used to rederive the
main results of \cite{demazure:inria-00075672}, and to rephrase the five-point
algorithm in terms of intersections of six bilinear forms in $\bP^3\times\bP^3$
instead of the nine Demazure cubics and five linear forms. This is also related
to the results of \cite{kileel}, but the derivations are completely different
and independent of \cite{demazure:inria-00075672} (which is used in an essential
way in \cite{kileel}).

\begin{question}
  What is the correct boundary for $\Cam_n$ (resp. $\CalCam_n$)?
\end{question}
Is there a extension of our moduli theory to handle degenerate configurations,
where camera centers collide? Should these models include degenerations of image
planes along the lines of Hacking's approach \cite{MR2078368}? Is there a good
moduli theory for pairs $(X,C)$ consisting of a threefold with an embedded
curve? These might be useful for studying degenerations of the ambient space
together with its calibrating curve.

\begin{question}
  What is the right general formulation of Carlsson--Weinshall duality?
\end{question}
Carlsson--Weinshall duality is somewhat mysterious from the point of view taken
here. One can think about it in terms of birational isomorphisms of universal
correspondences. It would be interesting to get a deeper understanding of this
phenomenon.


\begin{thebibliography}{10}

  \bibitem{aholt2011hilbert}
  {\sc C.~Aholt, B.~Sturmfels, and R.~Thomas}, {\em A {H}ilbert scheme in
    computer vision}, Canad. J. Math., 65 (2013), pp.~961--988,
    \url{https://doi.org/10.4153/CJM-2012-023-2},
    \url{http://dx.doi.org/10.4153/CJM-2012-023-2}.
  
  \bibitem{demazure:inria-00075672}
  {\sc M.~Demazure}, {\em Sur deux problemes de reconstruction}, Tech. Report
    RR-0882, INRIA, July 1988, \url{https://hal.inria.fr/inria-00075672}.
  
  \bibitem{kileel}
  {\sc G.~Fl{\o}ystad, J.~Kileel, and G.~Ottaviani}, {\em The {C}how form of the
    essential variety in computer vision}, Journal of Symbolic Computation,
    (2017), \url{https://doi.org/10.1016/j.jsc.2017.03.010}.
  \newblock To appear.
  
  \bibitem{MR0217085}
  {\sc A.~Grothendieck}, {\em \'{E}l\'{e}ments de g\'{e}om\'{e}trie
    alg\'{e}brique. {III}. \'{E}tude cohomologique des faisceaux coh\'{e}rents.
    {I}}, Inst. Hautes \'{E}tudes Sci. Publ. Math.,  (1961), p.~167,
    \url{http://www.numdam.org/item?id=PMIHES_1961__11__167_0}.
  
  \bibitem{MR2078368}
  {\sc P.~Hacking}, {\em Compact moduli of plane curves}, Duke Math. J., 124
    (2004), pp.~213--257, \url{https://doi.org/10.1215/S0012-7094-04-12421-2},
    \url{https://doi-org.offcampus.lib.washington.edu/10.1215/S0012-7094-04-12421-2}.
  
  \bibitem{hartley2003multiple}
  {\sc R.~Hartley and A.~Zisserman}, {\em Multiple view geometry in computer
    vision}, Cambridge university press, 2003.
  
  \bibitem{MR2047697}
  {\sc B.~Hassett and S.~J. Kov\'acs}, {\em Reflexive pull-backs and base
    extension}, J. Algebraic Geom., 13 (2004), pp.~233--247,
    \url{https://doi.org/10.1090/S1056-3911-03-00331-X},
    \url{http://dx.doi.org/10.1090/S1056-3911-03-00331-X}.
  
  \bibitem{MR1288306}
  {\sc W.~V.~D. Hodge and D.~Pedoe}, {\em Methods of algebraic geometry. {V}ol.
    {II}}, Cambridge Mathematical Library, Cambridge University Press, Cambridge,
    1994, \url{https://doi.org/10.1017/CBO9780511623899},
    \url{https://doi-org.offcampus.lib.washington.edu/10.1017/CBO9780511623899}.
  \newblock Book III: General theory of algebraic varieties in projective space,
    Book IV: Quadrics and Grassmann varieties, Reprint of the 1952 original.
  
  \bibitem{huybrechts2006fourier}
  {\sc D.~Huybrechts}, {\em Fourier-Mukai transforms in algebraic geometry},
    Oxford University Press, 2006.
  
  \bibitem{MR0491680}
  {\sc L.~Illusie}, {\em Complexe cotangent et d\'eformations. {I}}, Lecture
    Notes in Mathematics, Vol. 239, Springer-Verlag, Berlin-New York, 1971.
  
  \bibitem{kileel2017minimal}
  {\sc J.~Kileel}, {\em Minimal problems for the calibrated trifocal variety},
    SIAM Journal on Applied Algebra and Geometry, 1 (2017), pp.~575--598.
  
  \bibitem{essvar}
  {\sc M.~Lieblich, L.~Van~Meter, and B.~Viray}, {\em A new approach to the
    essential variety},  (2018).
  \newblock in preparation.
  
  \bibitem{MR1226451}
  {\sc S.~Maybank}, {\em Theory of reconstruction from image motion}, vol.~28 of
    Springer Series in Information Sciences, Springer-Verlag, Berlin, 1993,
    \url{https://doi.org/10.1007/978-3-642-77557-4},
    \url{http://dx.doi.org/10.1007/978-3-642-77557-4}.
  
  \bibitem{sernesi2007deformations}
  {\sc E.~Sernesi}, {\em Deformations of algebraic schemes}, vol.~334, Springer
    Science \& Business Media, 2007.
  
  \bibitem{stacks-project}
  {\sc T.~{Stacks Project Authors}}, {\em Stacks project}.
  \newblock \url{http://stacks.math.columbia.edu}, 2017.
  
  \bibitem{trager2015joint}
  {\sc M.~Trager, M.~Hebert, and J.~Ponce}, {\em The joint image handbook}, in
    Proceedings of the IEEE international conference on computer vision, 2015,
    pp.~909--917.
  
  \end{thebibliography}
\end{document}